\documentclass[11pt]{amsart}
\usepackage{amsmath,amsthm}
\usepackage{amstext,amsfonts,bm,amssymb}
\usepackage{graphics,graphicx}
\usepackage{a4wide}
\usepackage{float}
\usepackage{color}
\usepackage{verbatim}

\usepackage{gastex}

\theoremstyle{plain}
\newtheorem{theorem}{Theorem}
\newtheorem{Prop}{Proposition}

\newtheorem{Lem}{Lemma}

\def\fl#1{\left\lfloor#1\right\rfloor}

\def\stif#1#2{\left[#1\atop#2\right]} 
\def\sts#1#2{\left\{#1\atop#2\right\}}
\def\sttf2#1#2{\left[\!\!\left[#1\atop#2\right]\!\!\right]}  
\def\stf3f#1#2{\left[\!\!\left[\!\!\left[#1\atop#2\right]\!\!\right]\!\!\right]} 
\def\stff4#1#2{\left[\!\!\left[\!\!\left[\!\!\left[#1\atop#2\right]\!\!\right]\!\!\right]\!\!\right]}
\def\stss2#1#2{\left\{\!\!\left\{#1\atop#2\right\}\!\!\right\}}

\allowdisplaybreaks

\begin{document}

\title[A $q$-multiple zeta-star function at roots of unity]{Some explicit values of a $q$-multiple zeta-star function at roots of unity} 

\author{Takao Komatsu}
\address{Institute of Mathematics\\ Henan Academy of Sciences\\ Zhengzhou 450046\\ China;  \linebreak
Department of Mathematics, Institute of Science Tokyo, 2-12-1 Ookayama, Meguro-ku, Tokyo 152-8551, Japan}
\email{komatsu@hnas.ac.cn}

\date{%\today
}

\begin{abstract}
In this paper, we show some expressions of certain $q$-multiple zeta-star values at roots of unity.  These explicit formulas are expressed by using the determinants or Bell polynomials. Explicit formulas for other types of values can be found from recurrence relations obtained using generating functions.  
\medskip

\end{abstract}

\subjclass[2010]{Primary 11M32; Secondary 05A15, 05A19, 05A30, 11B37, 11B73}
\keywords{multiple zeta functions, $q$-Stirling numbers with higher level, complete homogeneous symmetric functions, Bell polynomials, determinant}

\maketitle
%%%%%%%%%%%%%%%%%%%%%%%%%%%%%%%%%%%%%%%%%%%%%%%%%%%%%%%%%%%%%%%%%%%%%%%%%%%%%%%

\section{Introduction}\label{sec:1}

In this paper, we consider the function 
\begin{equation}
\mathfrak Z_n^\star(q;m,s):=\sum_{1\le i_1\le i_2\le \dots\le i_m\le n-1}\frac{1}{(1-q^{i_1})^{s}(1-q^{i_2})^{s}\cdots(1-q^{i_m})^{s}}\,. 
\label{def:qomzv}
\end{equation} 
Then we show some fundamental ways to find the values, and give explicit expressions of the values of $\mathfrak Z_n^\star(\zeta_n;m,s)$ when $m$ or $s$ are comparatively small.    
When $n\to\infty$, the values in (\ref{def:qomzv}) may be called one of the $q$-multiple zeta-star values (see, e.g., \cite{OOZ,Takeyama09}).   

In previous papers (\cite{Ko24,Ko25}), as an application of the study of certain generalized Stirling numbers (\cite{Ko23,Ko24}), we considered a function  
\begin{equation}
\mathfrak Z_n(q;m,s):=\sum_{1\le i_1<i_2<\dots<i_m\le n-1}\frac{1}{(1-q^{i_1})^{s}(1-q^{i_2})^{s}\cdots(1-q^{i_m})^{s}}\,.   
\label{def:qmzv}
\end{equation}   
When $n\to\infty$, these values may be called one of the $q$-multiple zeta values (see, e.g., \cite{Bradley,Schlesinger,Tasaka21,Zhao,Zudilin}).  
Then when $q$ takes the root of unity $\zeta_n$ in (\ref{def:qmzv}), some explicit expressions can be given:  
\begin{align}
\mathfrak Z_n(\zeta_n;m,1)&=\frac{1}{m+1}\binom{n-1}{m}\,, 
\label{eq:zzm1}\\
\mathfrak Z_n(\zeta_n;m,2)&=\frac{1}{n(m+1)}\left(\binom{n-1}{m}+(-1)^m\binom{n-1}{2 m+1}\right) 
\label{eq:s2}\\  
&=\frac{2\cdot m!}{(2 m+2)!}\binom{n-1}{m}\sum_{k=0}^m\sttf2{2 m+2}{m+k+2}_1^{(m+1,1)}(-n)^k\,, 
\label{eq:ss2}
\end{align} 
where $\sttf2{n}{k}_1^{(r,1)}$ denotes the $r$-Stirling number (\cite{Broder}).  
We also have 
\begin{equation}
\mathfrak Z_n(\zeta_n;1,s)
=\left|\begin{array}{ccccc}
\frac{n-1}{2}&1&0&\cdots&\\ 
\frac{2}{3}\binom{n-1}{2}&\frac{n-1}{2}&1&&\vdots\\ 
\vdots&&\ddots&&0\\
\frac{s-1}{s}\binom{n-1}{s-1}&\frac{1}{s-1}\binom{n-1}{s-2}&\cdots&\frac{n-1}{2}&1\\ 
\frac{s}{s+1}\binom{n-1}{s}&\frac{1}{s}\binom{n-1}{s-1}&\cdots&\frac{1}{3}\binom{n-1}{2}&\frac{n-1}{2}\\ 
\end{array}
\right|\,.
\label{eq:zz-det2}
\end{equation} 
%And we also have 
%\begin{align}
%&\mathfrak Z_n(\zeta_n;m,3)\notag\\
%&=\frac{1}{n^2(m+1)}\left(\binom{n-1}{m}+\binom{n-1}{3 m+2}\right)\notag\\
%&\quad -\frac{1}{n^2}\sum_{k=0}^{\fl{\frac{m+1}{2}}}\sum_{i=0}^{m-2 k+1}\frac{1}{m-k+1}\binom{m-k+1}{k}\binom{m-2 k+1}{i}\binom{n+m-2 k-i}{3 m-3 k+2}2^i(-3)^{m-2 k-i+1}\,. 
%\label{eq:s3}
%\end{align} 
An expression of $\mathfrak Z_n(\zeta_n;m,s)$ ($s=3,4,\dots$) would be possible, but it looks to become very complicated.  

In addition, even though there is no direct expression of $\mathfrak Z_n(\zeta_n;m,s)$ ($m=2,3,\dots$), its value may be calculated by using the known ones.  

%Though the determinant expression in (\ref{eq:zz-det2}) is meaningful, 
%no direct expression of $\mathfrak Z_n(\zeta_n;1,s)$ has been known.  Nevertheless, it is possible to obtain each value of $\mathfrak Z_n(\zeta_n;m,s)$ by using the determinant relation 
%\begin{align*}
%&\mathfrak Z_n(\zeta_n;m,s)\\
%&=\frac{1}{m!}\mathbf Y_m\bigl(\mathfrak Z_n(\zeta_n;1,s),-1!\mathfrak Z_n(\zeta_n;1,2 s),%2!\mathfrak Z_n(\zeta_n;1,3 s),-3!\mathfrak Z_n(\zeta_n;1,4 s),\dots)\bigr)\\
%&=\frac{1}{m!}\left|\begin{array}{ccccc}
%\mathfrak Z_n(\zeta_n;1,s)&1&0&\cdots&\\ 
%\mathfrak Z_n(\zeta_n;1,2 s)&\mathfrak Z_n(\zeta_n;1,s)&2&&\vdots\\ 
%\vdots&&\ddots&&0\\
%\mathfrak Z_n\bigl(\zeta_n;1,(m-1)s\bigr)&\mathfrak Z_n\bigl(\zeta_n;1,(m-2)s\bigr)&\cdots&\mathfrak Z_n(\zeta_n;1,s)&m-1\\ 
%\mathfrak Z_n(\zeta_n;1,m s)&\mathfrak Z_n\bigl(\zeta_n;1,(m-1)s\bigr)&\cdots&\mathfrak Z_n(\zeta_n;1,2 s)&\mathfrak Z_n(\zeta_n;1,s)\\ 
%\end{array}
%\right|\,.
%\end{align*}
For example, when $m=2$, we can obtain 
$\mathfrak Z_n(\zeta_n;2,s)$ as  
\begin{equation}  
\mathfrak Z_n(\zeta_n;2,s)=\frac{1}{2}\left(\mathfrak Z_n(\zeta_n;1,s)^2-\mathfrak Z_n(\zeta_n;1,2 s)\right)\,.
\label{eq:m2s}
\end{equation} 
Similarly, for $m=3$,  
\begin{equation}  
\mathfrak Z_n(\zeta_n;3,s)=\frac{1}{6}\left(\mathfrak Z_n(\zeta_n;1,s)^3-3\mathfrak Z_n(\zeta_n;1,s)\mathfrak Z_n(\zeta_n;1,2 s)+2\mathfrak Z_n(\zeta_n;1,3 s)\right)\,.
\label{eq:m2s}
\end{equation}

%%%%%%%%%%%%%%%%%
%%%%%%%%%%%%%%%%%
%%%%%%%%%%%%%%%%%

It is noticed that, while many explicit formulas for the $q$-multiple zeta values are known, few explicit formulas for the $q$-multiple zeta-star values are known, even though the relationships involving the $q$-multiple zeta-star values are known (see, e.g., \cite{BTT18,BTT20,LP19}).  One of the reasons is that it is more difficult to handle the complete homogeneous symmetric functions corresponding to the  $q$-multiple zeta-star values than the elementary symmetric functions corresponding to the  $q$-multiple zeta values.

\section{$q$-generalized $(r,s)$-Stirling numbers} 

In \cite{Ko24}, $q$-generalized $(r,s)$-Stirling numbers of both kinds are introduced. These Stirling numbers are closely related to our $q$-multiple zeta function in (\ref{def:qmzv}) and $q$-multiple zeta-star function in (\ref{def:qomzv}).

Let $[n]_q$ denote the $q$-number, defined by 
$$
[n]_q=\frac{q^n-1}{q-1}\quad(q\ne 1)\,. 
$$  
Its $q$-factorial is given by $[n]_q!=[n]_q[n-1]_q\cdots[1]_q$ with $[0]_q!=1$. 
Let $r$ be a positive integer.    
The $q$-version of $r$-Stirling numbers of the first kind with higher level (level $s$) are denoted by $\sttf2{n}{k}_q^{(r,s)}$, and appear in the coefficients in the expansion 
\begin{equation}  
(x)_{n,q}^{(r,s)}=\sum_{k=0}^n(-1)^{n-k}\sttf2{n}{k}_q^{(r,s)}x^k\,,  
\label{def:qrsth1}  
\end{equation} 
where for $r,s\ge 1$, $(x)_{n,q}^{(r,s)}$ is defined by  
$$
(x)_{n,q}^{(r,s)}:=x^r\prod_{i=r}^{n-1}\bigl(x-([i]_q)^s\bigr)\quad(n>r)
$$ 
with $(x)_{r,q}^{(r,s)}=x^r$. 
When $r=s=1$, $s_q(n,k)=(-1)^{n-k}\sttf2{n}{k}_q^{(1,1)}$ are the signed $q$-Stirling numbers of the first kind (see, e.g., \cite{Ernst}), and $\sttf2{n}{k}_q=\sttf2{n}{k}_q^{(1,1)}$ are the unsigned $q$-Stirling numbers of the first kind. When $r=1$ and $q\to 1$, $\sttf2{n}{k}^{(s)}=\sttf2{n}{k}_1^{(1,s)}$ are the (unsigned) Stirling numbers of the first kind with higher level (\cite{KRV1}). When $r=s=1$ and $q\to 1$, $\stif{n}{k}=\sttf2{n}{k}_1^{(1,1)}$ are the unsigned Stirling numbers of the first kind.

The $q$-version of $r$-Stirling numbers of the second kind with higher level are denoted by $\stss2{n}{k}_q^{(r,s)}$, and appear in the coefficients in the expansion 
\begin{equation}  
x^n=\sum_{k=0}^n\stss2{n}{k}_q^{(r,s)}(x)_{k,q}^{(r,s)}\,.   
\label{def:qrsth2}  
\end{equation} 

When $r=1$ and $q\to 1$, $\stss2{n}{k}^{(s)}=\stss2{n}{k}_1^{(1,s)}$ are the Stirling numbers of the second kind with higher level, studied in \cite{KRV2}. 
When $r=s=1$, $S_q(n,k)=\stss2{n}{k}_q=(-1)^{n-k}\stss2{n}{k}_q^{(1,1)}$ are the signed $q$-Stirling numbers of the second kind (see, e.g., \cite{Ernst}). 
When $r=s=1$ and $q\to 1$, $\sts{n}{k}=\stss2{n}{k}_1^{(1,1)}$ are the classical Stirling numbers of the second kind.

From the definitions in \eqref{def:qrsth1} and \eqref{def:qrsth2}, the following recurrence relations hold.   
\begin{equation}  
\sttf2{n}{k}_q^{(r,s)}=\sttf2{n-1}{k-1}_q^{(r,s)}+\bigl([n-1]_q\bigr)^s\sttf2{n-1}{k}_q^{(r,s)}
\label{rec:qsth1} 
\end{equation}
with 
\begin{align*}  
&\sttf2{n}{k}_q^{(r,s)}=0\quad(0\le k\le r,\, n\ge k),\\ 
&\sttf2{n}{k}_q^{(r,s)}=0\quad(n<k),\quad \sttf2{0}{0}_q^{(r,s)}=1\,. 
\end{align*}
\begin{equation}  
\stss2{n}{k}_q^{(r,s)}=\stss2{n-1}{k-1}_q^{(r,s)}+\bigl([k]_q\bigr)^s\stss2{n-1}{k}_q^{(r,s)}
\label{rec:qsth2} 
\end{equation}
with 
\begin{align*}  
&\stss2{n}{k}_q^{(r,s)}=0\quad(0\le k\le r-1,~n\ge k),\quad \stss2{0}{0}_q^{(r,s)}=1,\\ 
&\stss2{n}{k}_q^{(r,s)}=0\quad(n\le k)\,. 
\end{align*}

By the definitions and/or the recurrence relations, we can get expressions with combinatorial summations (\cite{Ko24}).  

\begin{Lem}  
For $r\le m\le n-1$ and $r\ge 1$, we have 
$$ 
\sttf2{n}{m}_q^{(r,s)}=\left(\frac{[n-1]_q!}{[r-1]_q!}\right)^s\sum_{r\le i_1<\dots<i_{m-r}\le n-1}\frac{1}{\bigl([i_1]_q\dots[i_{m-r}]_q\bigr)^s}\,.
$$ 
For $n-m\ge r$ and $r\ge 1$, we have 
\begin{align*}
\sttf2{n}{n-m}_q^{(r,s)}&=\sum_{r\le i_1<\dots<i_m\le n-1}\bigl([i_1]_q\dots[i_m]_q\bigr)^s\\
&=\sum_{r\le i_1\le i_2\le\dots\le i_m\le n-m}([i_1]_q[i_2+1]_q\cdots[i_m+m-1]_q)^s\\
&=\sum_{i_m=r}^{n-m}([i_m+m-1]_q)^s\sum_{i_{m-1}=r}^{i_m}([i_{m-1}+m-2]_q)^s\cdots\sum_{i_{2}=r}^{i_3}([i_{2}+1]_q)^s\sum_{i_{1}=r}^{i_2}([i_{1}]_q)^s\,.
\end{align*}
\label{lem:exp11} 
\end{Lem}   

\begin{Lem}  
For $r+1\le k\le n$ and $r\ge 1$,  
\begin{multline*}
\stss2{n}{k}_q^{(r,s)}=\sum_{i_{k-r}=0}^{n-k}\bigl([k]_q\bigr)^{(n-k-i_{k-r})s}\sum_{i_{k-r-1}=0}^{i_{k-r}}\bigl([k-1]_q\bigr)^{(i_{k-r}-i_{k-r-1})s}\\
\cdots\sum_{i_2=0}^{i_3}\bigl([r+2]_q\bigr)^{(i_{3}-i_{2})s}\sum_{i_1=0}^{i_2}\bigl([r+1]_q\bigr)^{(i_{2}-i_{1})s}\bigl([r]_q\bigr)^{i_{1} s}\,. 
\end{multline*}
For $n-k\ge r\ge 1$, 
$$
\stss2{n}{n-k}_q^{(r,s)}=\sum_{r\le i_1\le i_2\le\cdots\le i_k\le n-k}([i_1]_q[i_2]_q\cdots[i_k]_q)^s\,. 
$$ 
\label{lem:iterated-sum} 
\end{Lem}

From Lemma \ref{lem:exp11} and Lemma \ref{lem:iterated-sum}, $q$-multiple zeta functions and $q$-multiple zeta-star functions can be written in terms of $q$-generalized $(1,s)$-Stirling numbers of the first kind and of the second kind, respectively.  

\begin{Prop}  
For positive integers $n$, $m$ and $s$, we have 
\begin{align*}
\mathfrak Z_n(q;m,s)&=\frac{1}{(1-q)^{m s}\bigl([n-1]_q!\bigr)^s}\sttf2{n}{m+1}_q^{(1,s)}\,,\\
\mathfrak Z_n^\star(q;m,s)&=\frac{1}{(1-q)^{m s}}\stss2{n+m-1}{n-1}_q^{(1,-s)}\,. 
\end{align*}
\label{prp:zeta-Stirling}
\end{Prop}

\section{The main result} 

Note that $\mathfrak Z_n^\star(q;1,s)=\mathfrak Z_n(q;1,s)$.  Hence, one can have interest in the case where $m\ge 2$.   
Then, in light of the result of $\mathfrak Z_n(q;m,s)$, when $q$ takes the root of unity,  we can get the following determinant expression.  

\begin{theorem} 
For integers $n,m$ with $n\ge 2$ and $m\ge 1$, we have 
\begin{align*}
&\mathfrak Z_n^\star(\zeta_n;m,s)\\
&=\frac{1}{m!}\mathbf Y_m\bigl(\mathfrak Z_n(\zeta_n;1,s),1!\mathfrak Z_n(\zeta_n;1,2 s),2!\mathfrak Z_n(\zeta_n;1,3 s),3!\mathfrak Z_n(\zeta_n;1,4 s),\dots)\bigr)\\
&=\frac{1}{m!}\left|\begin{array}{ccccc}
\mathfrak Z_n(\zeta_n;1,s)&-1&0&\cdots&\\ 
\mathfrak Z_n(\zeta_n;1,2 s)&\mathfrak Z_n(\zeta_n;1,s)&-2&&\vdots\\ 
\vdots&&\ddots&&0\\
\mathfrak Z_n\bigl(\zeta_n;1,(m-1)s\bigr)&\mathfrak Z_n\bigl(\zeta_n;1,(m-2)s\bigr)&\cdots&\mathfrak Z_n(\zeta_n;1,s)&-m+1\\ 
\mathfrak Z_n(\zeta_n;1,m s)&\mathfrak Z_n\bigl(\zeta_n;1,(m-1)s\bigr)&\cdots&\mathfrak Z_n(\zeta_n;1,2 s)&\mathfrak Z_n(\zeta_n;1,s)\\ 
\end{array}
\right|\,.
\end{align*} 
\label{th:mozv3}
\end{theorem}  

For example, when $m=2$, 
%by 
%$$
%\sum_{1\le i_1\le i_2\le n-1}\frac{1}{(1-q^{i_1})^{s_1}(1-q^{i_2})^{s_2}}=\sum_{1\le i_1<i_2\le n-1}\frac{1}{(1-q^{i_1})^{s_1}(1-q^{i_2})^{s_2}}+\sum_{i=1}^{n-1}\frac{1}{(1-q)^{s_1+s_2}}\,,
%$$
we have 
\begin{equation}  
\mathfrak Z_n^\star(\zeta_n;2,s)=\frac{1}{2}\left(\mathfrak Z_n(\zeta_n;1,s)^2+\mathfrak Z_n(\zeta_n;1,2 s)\right)\,.
\label{eq:mmo2}
\end{equation}
Another way is 
\begin{equation}  
\mathfrak Z_n^\star(\zeta_n;2,s)=\mathfrak Z_n(\zeta_n;2,s)+\mathfrak Z_n(\zeta_n;1,2 s)\,.
\label{eq:mo2}
\end{equation}

Then for $s=1,2,3$ we have 
\begin{align*}  
\mathfrak Z_n^\star(\zeta_n;2,1)&=\frac{(n+1)(n-1)}{12}\,,\\ 
\mathfrak Z_n^\star(\zeta_n;2,2)&=\frac{(n+1)(n-1)(n-3)(n-7)}{2\cdot 5!}\,,\\ 
\mathfrak Z_n^\star(\zeta_n;2,3)&=-\frac{(n+1)(n-1)(n^4-650 n^2+3780 n-5291)}{12\cdot 7!}\,.
\end{align*}

For example, when $m=2$ and $s=3$, by using the known values of $\mathfrak Z_n(\zeta_n;2,3)$ (or $\mathfrak Z_n(\zeta_n;1,3)$) and $\mathfrak Z_n(\zeta_n;1,6)$, we have 
\begin{align*}
&\mathfrak Z_n^\star(\zeta_n;2,3)=\frac{1}{2}\left(\mathfrak Z_n(\zeta_n;1,3)^2+\mathfrak Z_n(\zeta_n;1,6)\right)\\
&=\frac{1}{2}\left(\frac{(n-1)^2(n-3)^2}{8^2}-\frac{(n-1)(2 n^5+2 n^4-355 n^3-355 n^2+11153 n-19087)}{12\cdot 7!}\right)\\
&=-\frac{(n+1)(n-1)(n^4-650 n^2+3780 n-5291)}{12\cdot 7!} 
\end{align*}
or 
\begin{align*}
&\mathfrak Z_n^\star(\zeta_n;2,3)=\mathfrak Z_n(\zeta_n;2,3)+\mathfrak Z_n(\zeta_n;1,6)\\
&=\frac{6(n-1)(n-2)(n^4+3 n^3+301 n^2-2883 n+6898)}{9!}\\
&\quad -\frac{(n-1)(2 n^5+2 n^4-355 n^3-355 n^2+11153 n-19087)}{12\cdot 7!}\\
&=-\frac{(n+1)(n-1)(n^4-650 n^2+3780 n-5291)}{12\cdot 7!}\,.
\end{align*}

In order to prove Theorem \ref{th:mozv3}, we need the following Lemmas (\cite[Proposition 1]{Hoffman17},\cite[Ch.I\S 2]{MacDonald95}). 
Notice that in the case of (\ref{def:qmzv}), the equivalent of elementary symmetric functions arises, whereas in the case of (\ref{def:qomzv}), the equivalent of complete homogeneous symmetric functions arises.  

\begin{Lem} 
Let $n$, $M$ and $K$ be positive integers with $M\ge K$. 
For $g_n:=a_1^n+a_2^n+\cdots+a_M^n$ we have 
$$
\sum_{1\le j_1\le j_2\le \dots\le j_K\le M}a_{j_1}a_{j_2}\cdots a_{j_K}=\frac{1}{K!}\mathbf Y_K(g_1,1!g_2,2! g_3,3! g_4,\dots)\,, 
$$
where $\mathbf Y_n(x_1,x_2,\dots,x_n)$ is the (complete exponential) Bell polynomial, defined by 
$$
\exp\left(\sum_{m=1}^\infty x_m\frac{t^m}{m!}\right)=1+\sum_{n=1}^\infty\mathbf Y_n(x_1,x_2,\dots,x_n)\frac{t^n}{n!}\,. 
$$ 
That is, 
\begin{align*}
&\mathbf Y_n(x_1,x_2,x_3,\dots,x_n)\\ 
&=\sum_{k=1}^n\sum_{i_1+2 i_2+\cdots+(n-k+1)i_{n-k+1}=n\atop i_1+i_2+i_3+\cdots=k}\frac{n!}{i_1!i_2!\cdots i_{n-k+1}!}\left(\frac{x_1}{1!}\right)^{i_1}\left(\frac{x_2}{2!}\right)^{i_2}\cdots\left(\frac{x_{n-k+1}}{(n-k+1)!}\right)^{i_{n-k+1}}
\end{align*} 
with $\mathbf Y_0=1$. 
\label{lem:obell}
\end{Lem}

\begin{Lem}  
The following expressions are equivalent.  
\begin{enumerate}
\item[(i)] $\displaystyle b_m=\sum_{i_1+2 i_2+\cdots+m i_m=m\atop i_1,i_2,\dots,i_m\ge 0}\frac{1}{i_1!i_2!\cdots i_m!}\left(\frac{a_1}{1}\right)^{i_1}\left(\frac{a_2}{2}\right)^{i_2}\cdots\left(\frac{a_m}{m}\right)^{i_m}$
\item[(ii)] $\displaystyle b_m=\frac{1}{m!}\left|\begin{array}{ccccc}
a_1&-1&0&\cdots&0\\
a_2&a_1&-2&&\vdots\\
\vdots&&\ddots&&0\\
a_{m-1}&a_{m-2}&\cdots&a_1&-m+1\\
a_m&a_{m-1}&\cdots&a_2&a_1\\
\end{array}
\right|$ 
\item[(iii] $\displaystyle a_n=(-1)^{n-1}\left|\begin{array}{ccccc}
b_1&1&0&\cdots&0\\
2 b_2&b_1&1&&\vdots\\
\vdots&&\ddots&&0\\
(n-1)b_{n-1}&b_{n-2}&\cdots&b_1&1\\
n b_n&b_{n-1}&\cdots&b_2&b_1\\
\end{array}
\right|$ 
\item[(iv)] $\displaystyle m b_m=\sum_{i=1}^m a_i b_{m-i}$
\item[(v)] $\displaystyle a_n=-\sum_{j=1}^{n-1} b_j a_{n-j}+n b_n$ 
\end{enumerate}
\label{lem:gotrudi}
\end{Lem}

\noindent 
{\it Remark.}  
The relations between  (i) and (ii) are yielded from a simple modification of Trudi's formula (\cite[Vol.3, p.214]{Muir},\cite{Trudi}). 
By applying the inversion formula  (see, e.g., \cite[Lemma 1]{Ko25a},\cite[Theorem 1]{KR}), we can find that (iv) and (v) are equivalent.  
Others can be yielded by the expansions of the determinants.

\begin{proof}[Proof of Theorem \ref{th:mozv3}.]  
Put $b_m=\mathfrak Z_n^\star(\zeta_n;m,s)$ and $a_j=\mathfrak Z_n(\zeta_n;1,j s)$ in Lemma \ref{lem:gotrudi} by applying Lemma \ref{lem:obell}.  
\end{proof}

\section{The values $\mathfrak Z_n^\star(\zeta_n;m,s)$ for small $s$}  

In the previous section, we considered the values $\mathfrak Z_n^\star(\zeta_n;m,s)$ for small $m$.  In this section we shall consider the values $\mathfrak Z_n^\star(\zeta_n;m,s)$ for small $s$.     

$\mathfrak Z_n(\zeta_n;m,1)$ is expressed in a simple form as in (\ref{eq:zzm1}), but $\mathfrak Z_n^\star(\zeta_n;m,1)$ is not.  
\begin{align*}
\mathfrak Z_n^\star(\zeta_n;1,1)&=\frac{n-1}{2}\,,\\ 
\mathfrak Z_n^\star(\zeta_n;2,1)&=\frac{(n+1)(n-1)}{12}\,,\\ 
\mathfrak Z_n^\star(\zeta_n;3,1)&=\frac{(n+1)(n-1)}{24}\,,\\ 
\mathfrak Z_n^\star(\zeta_n;4,1)&=-\frac{(n+1)(n-1)(n^2-19)}{720}\,,\\ 
\mathfrak Z_n^\star(\zeta_n;5,1)&=-\frac{(n+1)(n-1)(n+3)(n-3)}{480}\,,\\ 
\mathfrak Z_n^\star(\zeta_n;6,1)&=\frac{(n+1)(n-1)(2 n^4-145 n^2+863)}{60480}\,,\\ 
\mathfrak Z_n^\star(\zeta_n;7,1)&=\frac{(n+1)(n-1)(n+5)(n-5)(2 n^2-11)}{24192}\,,\\ 
\mathfrak Z_n^\star(\zeta_n;8,1)&=-\frac{(n+1)(n-1)(3 n^6-497 n^4+9247 n^2-33953)}{3628800}\,,\\
\mathfrak Z_n^\star(\zeta_n;9,1)&=-\frac{(n+1)(n-1)(n+7)(n-7)(3 n^4-50 n^2+167)}{1036800}\,.
\end{align*}
As $\mathfrak Z_n^\star(\zeta_n;m,s)$ is a complete homogeneous symmetric polynomial, let us consider the generating function:  
\begin{align*}  
&\sum_{m=0}^{n-1}\mathfrak Z_n^\star(\zeta_n;m,s)X^m=\sum_{m=0}^\infty\mathfrak Z_n^\star(\zeta_n;m,s)X^m\\
&=\sum_{m=0}^\infty\sum_{1\le i_1\le i_2\le \dots\le i_m\le n-1}\frac{X^m}{(1-\zeta_n^{i_1})^{s}(1-\zeta_n^{i_2})^{s}\cdots(1-\zeta_n^{i_m})^{s}}\\
&=\prod_{j=1}^{n-1}\dfrac{1}{1-\dfrac{X}{(1-\zeta_n^j)^s}}
=n^s\prod_{j=1}^{n-1}\frac{1}{(1-\zeta_n^j)^s-X}\,.  
\end{align*}  
Let $\alpha_i$ ($i=1,2,\dots,s$) be the roots of the polynomial $(1-Y)^s-X$. Namely, 
\begin{equation}
\prod_{i=1}^s(\alpha_i-Y)=(1-Y)^s-X\,. 
\label{eq:alpha}
\end{equation}
Since by (\ref{eq:alpha}) 
$$
\prod_{i=1}^s(\alpha_i-1)=-X\,,  
$$ 
we get 
$$
\prod_{j=1}^{n-1}\prod_{i=1}^s(\alpha_i-\zeta_n^j)=\prod_{i=1}^s\frac{\alpha_i^n-1}{\alpha_i-1}=\frac{-1}{X}\prod_{i=1}^s(\alpha_i^n-1)\,. 
$$ 
Hence, 
\begin{equation}
\sum_{m=0}^{n-1}\mathfrak Z_n^\star(\zeta_n;m,s)X^m=-X n^s\prod_{i=1}^s\frac{1}{\alpha_i^n-1}\,. 
\label{eq:qmzsv11}
\end{equation}

\subsection{The values $\mathfrak Z_n^\star(\zeta_n;m,1)$}  

Let us consider the value $\mathfrak Z_n^\star(\zeta_n;m,s)$ when $s=1$. 
When $s=1,2$, it is simpler and faster to calculate the values directly, though there is a general method that is described in Subsection \ref{sub3.3}.  
  
When $s=1$, by $\alpha_1-Y=(1-Y)-X$ in (\ref{eq:alpha}), we see that $\alpha_1=1-X$.  
Thus, by (\ref{eq:qmzsv11}) we have 
\begin{align*}
\sum_{m=0}^{n-1}\mathfrak Z_n^\star(\zeta_n;m,1)X^m
&=n\prod_{j=1}^{n-1}\frac{1}{(1-\zeta_n^j)-X}\\
&=\frac{-X n}{\alpha_1^n-1}=\frac{X n}{1-(1-X)^n}\\
&=\frac{n}{\sum_{\kappa=0}^{n-1}(1-X)^\kappa}=\frac{n}{\sum_{k=0}^{n-1}\binom{n}{k+1}(-1)^{k}X^k}\\
&=\frac{1}{\sum_{k=0}^{n-1}\binom{n-1}{k}\frac{(-1)^{k}}{k+1}X^k}=\frac{1}{1-\sum_{k=1}^{n-1}\binom{n-1}{k}\frac{(-1)^{k+1}}{k+1}X^k}  
\end{align*}  
Hence, 
\begin{align*}
1&=\left(\sum_{k=0}^{n-1}\binom{n-1}{k}\frac{(-1)^{k}}{k+1}X^k\right)\left(\sum_{m=0}^{n-1}\mathfrak Z_n^\star(\zeta_n;m,1)X^m\right)\\
&=\sum_{l=0}^{n-1}\left(\sum_{j=0}^l\frac{(-1)^{l-j}}{l-j+1}\binom{n-1}{l-j}\mathfrak Z_n^\star(\zeta_n;j,1)\right)X^l\\
&\quad +\sum_{l=n}^{2 n-2}\left(\sum_{j=l-n+1}^{n-1}\frac{(-1)^{l-j}}{l-j+1}\binom{n-1}{l-j}\mathfrak Z_n^\star(\zeta_n;j,1)\right)X^l\,. 
\end{align*}  
Comparing the coefficients on both sides, we have the recurrence formula when $s=1$.  By applying this formula repeatedly, we can get explicit expressions of $\mathfrak Z_n^\star(\zeta_n;m,1)$ ($m=1,2,\dots,9$) as seen above.  

\begin{theorem}  
For $1\le m\le n-1$, we have 
$$
\mathfrak Z_n^\star(\zeta_n;m,1)=\sum_{j=0}^{m-1}\frac{(-1)^{m-j+1}}{m-j+1}\binom{n-1}{m-j}\mathfrak Z_n^\star(\zeta_n;j,1)
$$ 
with $\mathfrak Z_n^\star(\zeta_n;0,1)=1$.  
\label{th:qmozeta1}
\end{theorem}

\subsection{The values $\mathfrak Z_n^\star(\zeta_n;m,2)$}  

Let us consider the value $\mathfrak Z_n^\star(\zeta_n;m,s)$ when $s=2$.   
When $s=2$, by (\ref{eq:alpha}), we see that $\alpha_1+\alpha_2=2$ and $\alpha_1\alpha_2=1-X$.  
Thus, by (\ref{eq:qmzsv11}) we have 
\begin{align*}
\sum_{m=0}^{n-1}\mathfrak Z_n^\star(\zeta_n;m,2)X^m
&=n^2\prod_{j=1}^{n-1}\frac{1}{(1-\zeta_n^j)^2-X}\\
&=\frac{-X n^2}{(\alpha_1^n-1)(\alpha_2^n-1)}=\frac{-X n^2}{1-\bigl((1+\sqrt{X})^n+(1-\sqrt{X})^n\bigr)+(1-X)^n}\\
&=\frac{-X n^2}{1-2\sum_{\kappa=0}^{\fl{\frac{n}{2}}}\binom{n}{2\kappa}X^\kappa+\sum_{\kappa=0}^{n}\binom{n}{\kappa}(-X)^\kappa}\\  
&=\frac{-n^2}{-2\sum_{\kappa=1}^{\fl{\frac{n}{2}}}\binom{n}{2\kappa}X^{\kappa-1}+\sum_{\kappa=1}^{n}\binom{n}{\kappa}(-1)^\kappa X^{\kappa-1}}\\  
&=\frac{n^2}{2\sum_{\kappa=0}^{\fl{\frac{n}{2}}-1}\binom{n}{2\kappa+2}X^\kappa+\sum_{\kappa=0}^{n-1}\binom{n}{\kappa+1}(-1)^\kappa X^\kappa}\,.  
\end{align*}
Hence, 
\begin{align*}
1&=\left(\frac{2}{n^2}\sum_{\kappa=0}^{\fl{\frac{n}{2}}-1}\binom{n}{2\kappa+2}X^\kappa\right)\left(\sum_{m=0}^{n-1}\mathfrak Z_n^\star(\zeta_n;m,2)X^m\right)\\
&\quad+\left(\frac{1}{n^2}\sum_{\kappa=0}^{n-1}\binom{n}{\kappa+1}(-1)^\kappa X^\kappa\right)\left(\sum_{m=0}^{n-1}\mathfrak Z_n^\star(\zeta_n;m,2)X^m\right)\\
&=\sum_{l=0}^{\fl{\frac{n}{2}}-1}\left(\sum_{j=0}^l\frac{2}{n^2}\binom{n}{2 l-2 j+2}\mathfrak Z_n^\star(\zeta_n;j,2)\right)X^l\\
&\quad+\sum_{l=0}^{n-1}\left(\sum_{j=0}^l\frac{(-1)^{l-j}}{n^2}\binom{n}{l-j+1}\mathfrak Z_n^\star(\zeta_n;j,2)\right)X^l\\
&\quad +R(n,l)\,,
\end{align*}  
where $R(n,l)$ implies the terms for $l\ge\fl{\frac{n}{2}}$ and for $l\ge n$, respectively.  
Comparing the coefficients on both sides, we have for $1\le l\le\fl{\frac{n}{2}}-1$ 
$$
\sum_{j=0}^l\frac{1}{n^2}\left(2\binom{n}{2 l-2 j+2}+(-1)^{l-j}\binom{n}{l-j+1}\right)\mathfrak Z_n^\star(\zeta_n;j,2)=0\,. 
$$
Hence, we have the following.  

\begin{theorem}  
For $1\le m\le\fl{\frac{n}{2}}-1$, we have 
$$
\mathfrak Z_n^\star(\zeta_n;m,2)=-\sum_{j=0}^{m-1}\frac{1}{n^2}\left(2\binom{n}{2 m-2 j+2}+(-1)^{m-j}\binom{n}{m-j+1}\right)\mathfrak Z_n^\star(\zeta_n;j,2)
$$ 
with $\mathfrak Z_n^\star(\zeta_n;0,2)=1$.   
\label{th:qmozeta2}
\end{theorem}

By applying this formula repeatedly, we can get explicit expressions of $\mathfrak Z_n^\star(\zeta_n;m,2)$ ($m=1,2,3,4,5$) as follows.  
\begin{align*} 
\mathfrak Z_n^\star(\zeta_n;1,2)&=-\frac{(n-1)(n-5)}{12}\,,\\
\mathfrak Z_n^\star(\zeta_n;2,2)&=\frac{(n+1)(n-1)(n-3)(n-7)}{240}\,,\\
\mathfrak Z_n^\star(\zeta_n;3,2)&=-\frac{(n+1)(n-1)(10 n^4-126 n^3+241 n^2+1134 n-2699)}{12\cdot 7!}\,,\\
\mathfrak Z_n^\star(\zeta_n;4,2)&=\frac{(n+1)(n-1)(21 n^6-300 n^5+181 n^4+9150 n^3-21071 n^2-41250 n+103669)}{10!}\\
\mathfrak Z_n^\star(\zeta_n;5,2)&=-\frac{(n+1)(n-1)(n-3)}{4\cdot 11!}\times(30 n^7-372 n^6-1889 n^5\\
&\qquad +24671 n^4+18790 n^3-346648 n^2-57251 n+1088429)\,. 
\end{align*}

\subsection{The cases for $s\ge 3$}\label{sub3.3}   

It is not easy to find an explicit expression of $\mathfrak Z_n^\star(\zeta_n;m,s)$ when $s\ge 3$ though it is possible theoretically.   

The following method was hinted by Henrik Bachmann directly when the author visited his university in 2024 (\cite{BTT18,BTT20}).   
In order to calculate $\prod_{i=1}^s(\alpha_i^n-1)^{-1}$, we consider the function 
\begin{align*}  
&f(X,Y):=\sum_{n=1}^\infty\frac{Y^n}{n}\prod_{i=1}^s(\alpha_i^n-1)\\
&=\sum_{n=1}^\infty\frac{Y^n}{n}\sum_{l=0}^s(-1)^{s-l}\sum_{1\le i_1<\cdots<i_l\le s}(\alpha_{i_1}\cdots\alpha_{i_l})^n\\ 
&=(-1)^{s-1}\sum_{l=0}^s(-1)^{l-1}\sum_{1\le i_1<\cdots<i_l\le s}\sum_{n=1}^\infty\frac{(\alpha_{i_1}\cdots\alpha_{i_l}Y)^n}{n}\\ 
&=(-1)^{s-1}\sum_{l=0}^s(-1)^{l}\sum_{1\le i_1<\cdots<i_l\le s}\log(1-\alpha_{i_1}\cdots\alpha_{i_l}Y)\\
&=(-1)^{s-1}\log\left(\prod_{l=0}^s F_{s,l}(X,Y)^{(-1)^l}\right)\,,
\end{align*}  
where 
\begin{equation}  
F_{s,l}(X,Y):=\prod_{1\le i_1<\cdots<i_l\le s}(1-\alpha_{i_1}\cdots\alpha_{i_l}Y)\quad(1\le l\le s) 
\label{def:ff}
\end{equation}
with $F_{s,0}(X,Y)=1-Y$.   
Here, $F_{s,l}(X,Y)$'s include the elementary symmetric functions but can be determined from (\ref{eq:alpha}) as the complete homogeneous symmetric functions. 

For example, we can find that 
\begin{align*}
F_{1,1}(X,Y)&=1-Y+X Y\,,\\ 
F_{2,1}(X,Y)&=(1-Y)^2-X Y^2\,,\\
F_{2,2}(X,Y)&=1-Y+X Y\,,\\ 
F_{3,1}(X,Y)&=(1-Y)^3+X Y^3\,,\\
F_{3,2}(X,Y)&=(1-Y)^3-3 X Y^2-(X^2-2 X)Y^3\,,\\
F_{3,3}(X,Y)&=1-Y+X Y\,. 
\end{align*}
Though it becomes more complicated, using $F_{1,1}(X,Y)$, or $F_{2,1}(X,Y)$ and $F_{2,2}(X,Y)$, we can find the expression $\prod_{i=1}^s(\alpha_i^n-1)$ for $s=1$ or $s=2$, respectively. 

Let us find the expression $\prod_{i=1}^s(\alpha_i^n-1)$ for $s=3$ by using $F_{3,1}(X,Y)$, $F_{3,2}(X,Y)$ and $F_{3,3}(X,Y)$. 

\begin{align*}  
&f(X,Y)=(-1)^{3-1}\log\left(\prod_{l=0}^3 F_{3,l}(X,Y)^{(-1)^l}\right)\\
&=\log\frac{(1-Y)\bigl((1-Y)^3-3 X Y^2-(X^2-2 X)Y^3\bigr)}{\bigl((1-Y)^3+X Y^3\bigr)(1-Y+X Y)}\\
&=\log\left(1-\frac{3 X Y^2+(X^2-2 X)Y^3}{(1-Y)^3}\right)-\log\left(1+\frac{X Y^3}{(1-Y)^3}\right)-\log\left(1+\frac{X Y}{1-Y}\right)\\
&=-\sum_{n=1}^\infty\frac{1}{n}\left(\frac{3 X Y^2+(X^2-2 X)Y^3}{(1-Y)^3}\right)^n-\sum_{n=1}^\infty\frac{(-1)^{n+1}}{n}\left(\frac{X Y^3}{(1-Y)^3}\right)^n-\sum_{n=1}^\infty\frac{(-1)^{n+1}}{n}\left(\frac{X Y}{1-Y}\right)^n\\
&=-\sum_{n=1}^\infty\frac{X^n Y^{2 n}\bigl(3+(X-2)Y\bigr)^n}{n}\sum_{j=0}^\infty\binom{3 n+j-1}{3 n-1}Y^j\\
&\quad +\sum_{n=1}^\infty\frac{(-1)^n X^n Y^{3 n}}{n}\sum_{j=0}^\infty\binom{3 n+j-1}{3 n-1}Y^j+\sum_{n=1}^\infty\frac{(-1)^n X^n Y^{n}}{n}\sum_{j=0}^\infty\binom{n+j-1}{n-1}Y^j\\
&=-\sum_{n=1}^\infty\frac{X^n Y^{2 n}}{n}\sum_{k=0}^n\binom{n}{k}3^{n-k}(X-2)^k Y^k\sum_{j=0}^\infty\binom{3 n+j-1}{3 n-1}Y^j\\
&\quad +\sum_{n=1}^\infty\frac{(-1)^n X^n}{n}\sum_{\nu=3 n}^\infty\binom{\nu-1}{3 n-1}Y^\nu+\sum_{n=1}^\infty\frac{(-1)^n X^n}{n}\sum_{\nu=n}^\infty\binom{\nu-1}{n-1}Y^\nu\\
&=-\sum_{\nu=2}^\infty Y^\nu\sum_{\mu=1}^\infty\sum_{k=0}^\mu\frac{X^{\mu}}{\mu}\binom{\mu}{k}3^{\mu-k}\sum_{l=0}^k\binom{k}{l}X^j(-2)^{k-l}\binom{\nu+\mu-k-1}{3\mu-1}\\
&\quad +\sum_{\nu=3}^\infty Y^\nu\sum_{m=1}^{\fl{\frac{\nu}{3}}}\binom{\nu-1}{3 m-1}\frac{(-1)^m X^m}{m}+\sum_{\nu=1}^\infty Y^\nu\sum_{m=1}^{\nu}\binom{\nu-1}{m-1}\frac{(-1)^m X^m}{m}\\
&=-\sum_{\nu=2}^\infty Y^\nu\sum_{m=1}^\infty X^m\sum_{l=0}^{\fl{\frac{m}{2}}}\sum_{k=l}^{m-l}\frac{3^{m-l-k}(-2)^{k-l}}{m-l}\binom{m-l}{k}\binom{k}{l}\binom{\nu+m-l-k-1}{3 m-3 l-1}\\
&\quad +\sum_{n=3}^\infty Y^n\sum_{m=1}^{\fl{\frac{n}{3}}}\binom{n-1}{3 m-1}\frac{(-1)^m X^m}{m}+\sum_{n=1}^\infty Y^n\sum_{m=1}^{n}\binom{n-1}{m-1}\frac{(-1)^m X^m}{m}\,.
\end{align*}  
Comparing the coefficients, we have 
\begin{align*}  
&\prod_{i=1}^3(\alpha_i^n-1)\\
&=n\sum_{m=0}^{\fl{\frac{n}{3}}-1}\binom{n-1}{3 m+2}\frac{(-1)^{m+1}X^{m+1}}{m+1}+n\sum_{m=0}^{n-1}\binom{n-1}{m}\frac{(-1)^{m+1}X^{m+1}}{m+1}\\
&\quad -n\sum_{m=0}^\infty X^{m+1}\sum_{l=0}^{\fl{\frac{m+1}{2}}}\sum_{k=l}^{m-l+1}\frac{3^{m-l-k+1}(-2)^{k-l}}{m-l+1}\binom{m-l+1}{k}\binom{k}{j}\binom{n+m-l-k}{3 m-3 l+2}\,.
\end{align*}  

Thus, by (\ref{eq:qmzsv11}) we have 
\begin{align*}  
n^2&=\left(\sum_{m=0}^{n-1}\mathfrak Z_n^\star(\zeta_n;m,3)X^m\right)\biggl(\sum_{m=0}^{\fl{\frac{n}{3}}-1}\binom{n-1}{3 m+2}\frac{(-1)^{m}X^{m}}{m+1}+\sum_{m=0}^{n-1}\binom{n-1}{m}\frac{(-1)^{m}X^{m}}{m+1}\\
&\qquad +\sum_{m=0}^\infty X^{m}\sum_{l=0}^{\fl{\frac{m+1}{2}}}\sum_{k=l}^{m-l+1}\frac{3^{m-l-k+1}(-2)^{k-l}}{m-l+1}\binom{m-l+1}{k}\binom{k}{l}\binom{n+m-l-k}{3 m-3 l+2}\biggr)\,.
\end{align*}  
Hence, we have the following.   

\begin{theorem}  
For $1\le m\le\fl{\frac{n}{3}}-1$, we have 
\begin{align*}  
&\mathfrak Z_n^\star(\zeta_n;m,3)\\
&=-\sum_{j=0}^{m-1}\frac{1}{n^2}\biggl(\sum_{l=0}^{\fl{\frac{m-j+1}{2}}}\sum_{k=l}^{m-j-l+1}\frac{3^{m-j-l-k+1}(-2)^{k-l}}{m-j-l+1}\binom{m-j-l+1}{k}\binom{k}{l}\binom{n+m-j-l-k}{3 m-3 j-3 l+2}\\
&\qquad +\left(\binom{n-1}{3 m-3 j+2}+\binom{n-1}{m-j}\right)\frac{(-1)^{m-j}}{m-j+1}\biggr)\mathfrak Z_n^\star(\zeta_n;j,3)
\end{align*}
with $\mathfrak Z_n^\star(\zeta_n;0,3)=1$.   
\label{th:qmozeta3}
\end{theorem}

By applying this formula repeatedly, we can get explicit expressions of $\mathfrak Z_n^\star(\zeta_n;m,3)$ ($m=1,2,3,4$) as follows.  
\begin{align*} 
\mathfrak Z_n^\star(\zeta_n;1,3)&=-\frac{(n-1)(n-3)}{8}\,,\\
\mathfrak Z_n^\star(\zeta_n;2,3)&=-\frac{(n+1)(n-1)(n^4-650 n^2+3780 n-5291)}{12\cdot 7!}\,,\\
\mathfrak Z_n^\star(\zeta_n;3,3)&=\frac{(n+1)(n-1)(n-3)(8 n^5+4 n^4-2005 n^3+6985 n^2+11357 n-36509)}{60\cdot 8!}\,,\\
\mathfrak Z_n^\star(\zeta_n;4,3)&=\frac{(n+1)(n-1)}{2\cdot 15!}(703 n^{10}-1109042 n^8+4324320 n^7+165698599 n^6-1101079980 n^5\\
&\quad -545336011 n^4+15460525080 n^3-18165129202 n^2-47055628620 n+76385318153)\,. 
\end{align*}

%\section*{Competing Interests}  
%The author declares no competing interest.

\section*{Acknowledgement}  

%The author thanks the referee for carefully reading the manuscript and for giving constructive comments. This work was partly done when the author stayed at Tongji University in September 2024. He is grateful for the generous hospitality of Professor Haigang Zhou and the fellowship and discussion by faculty members and students, in particular Professor Zhong-Hua Li. 
This work was partly supported by JSPS KAKENHI Grant Number 24K22835.

\end{document}